\documentclass[11pt]{article}
\usepackage[latin1]{inputenc}
\usepackage{epsfig}
\usepackage{color}
\usepackage[british,english]{babel}
\usepackage{amsthm}
\usepackage{amsmath}
\usepackage{amsfonts}
\usepackage{amssymb}
\usepackage{graphicx}
\newtheorem{corollary}{Corollary}[section]

\newtheorem{lemma}[corollary]{Lemma}
\newtheorem{proposition}[corollary]{Proposition}
\newtheorem{remark}[corollary]{Remark}
\newtheorem{theorem}[corollary]{Theorem}
\newfont{\sBlackboard}{msbm10 scaled 900}

\newcommand{\mylabel}[1]{\label{#1}
            \ifx\undefined\stillediting
            \else \fbox{$#1$}\fi }
\newcommand{\BE}{\begin{equation}}

\newcommand{\EEQ}{\end{equation}}
\newcommand{\rfb}[1]{\mbox{\rm
   (\ref{#1})}\ifx\undefined\stillediting\else:\fbox{$#1$}\fi}

\newfont{\Blackboard}{msbm10 scaled 1200}

\newfont{\roma}{cmr10 scaled 1200}

\def\CC{\rm \hbox{C\kern-.56em\raise.4ex
         \hbox{$\scriptscriptstyle |$}\kern+0.5 em }}

\def\Lim{\displaystyle\lim}

\def\Inf{\displaystyle\inf}

\def\Sum{\displaystyle\sum}
\def\Frac{\displaystyle\frac}
\def\Int{\displaystyle\int}
\def\n{|\kern -.05cm{|}\kern -.05cm{|}}

\def\R{{\bf \hbox{\sc I\hskip -2pt R}}} %rels
\def\N{\rm I\hskip -2pt N} %naturels
\newcommand{\mm}    {{\hbox{\hskip 0.5pt}}}

\newcommand{\bluff} {{\hbox{\raise 15pt \hbox{\mm}}}}
\def\df{\begin{flushright}$\Box$\end{flushright}}
\newcommand{\e}      {{\varepsilon}}

\makeatletter
\def\section{\@startsection {section}{1}{\z@}{-3.5ex plus -1ex minus
    -.2ex}{2.3ex plus .2ex}{\large\bf}}
\makeatother
\def\be{\begin{equation}}
\def\ee{\end{equation}}

\date{ }
\begin{document}\thispagestyle{empty}
\title{\bf Pointwise controllability as limit of internal controllability for the beam equation}\maketitle

%\author{ \center  Mama ABDELLI (1, 2)\\ 1.  Université Djillali Liabés, Laboratoire de Mathématique, \\ B.P.89. Sidi Bel Abbés 22000, Algeria..\\{2. Université de Mascara, B.P. 763. Mascara 29000, Algeria.\\
 %abdelli$_{-}$mama@yahoo.fr\\ Akram Ben Aissa\\
%\small Département de Mathématiques, Faculté des Sciences de Monastir,\\ \small Université de Monastir, 5019 Monastir, Tunisie\\
%\small Unité de recherche: Analyse et Contrôle des Equations aux Dérivées Partielles\\
%\small -ACEDP- 05/UR/15-01.\\
%\small \texttt{issaakram26@gmail.com}}
%\pagestyle{plain}
\begin{tabular}{c}
\sc Mama Abdelli $^{1}$, Akram Ben Aissa $^{2}$  \ \\
\small $^{1}$ Universit\'e Djillali Liab\'es,
\small   Laboratoire de Math\'ematiques,\qquad \\
\small B. P. 89, Sidi Bel Abb\'es 22000, ALGERIA.  \\e-mail:
abdelli.mama@gmail.com\\ \small $^{2}$ D\'epartement de
Math\'ematiques, Facult\'e des Sciences de
Monastir,\qquad  \\
\small 5019 Monastir, TUNISIA.\\e-mail: akram.benaissa@fsm.rnu.tn \\

\end{tabular}

\vskip20pt

 \renewcommand{\abstractname} {\bf Abstract}
 \begin{abstract}
This work is devoted to prove the pointwise controllability of the
  Bernoulli-Euler beam equation. It is obtained as a limit of internal controllability of  the same type of equation.
Our approach is based on the techniques used in {\bf\cite{FP1}}.
\end{abstract}
\bigskip\noindent
\\
%\\
%\\
%\\

{\small \noindent {\bf Keywords}: Beam equation, Internal control,
Exact controllability, Pointwise control.\\
{\small \noindent {\bf AMS Subject Classifications}: 35E15, 93D05, 93B07, 93D20.
\section {Introduction}
In this paper, we are interested in  the passage from internal  exact controllability of beam equation to pointwise exact controllability.
We consider the following initial
and boundary value problem
\begin{equation}\label{p1}
\Frac{\partial ^2 u}{\partial
t^2}(x,t)+\Frac{\partial^{4}u}{\partial x^4}(x,t)= g_n(x,t),\,\,\,
0<x<1,\,\,\,\, t>0,
\end{equation}
\begin{equation}\label{p2}
u(0,t)= \Frac{\partial u}{\partial x}(1,t)= \Frac{\partial^2
u}{\partial x^2}(0,t)= \Frac{\partial^3 u}{\partial
x^3}(1,t)=0,\quad t>0,
\end{equation}
\begin{equation}\label{p3}
u(x,0)=u^{0}(x),\,\,\,\,\ \Frac{\partial u}{\partial
t}(x,0)=u^{1}(x),\,\,\ 0<x<1,\end{equation} where $g_n,\,\,u^0,\,\,
u^1$ are in suitable spaces with
$\mathrm{supp}\,(g_n)=[\xi,\xi+\frac{1}{n}]$, $n\in\mathbb{N}^*$ and
$\xi \in (0,1)$.\\
 Here $u$ denotes the transverse displacement of
the beam, we suppose that the length of the beam is equal to $1$ and
the control depends on a parameter $n\in\N^*$. Recall that this
model describes the transversal vibrations of the Bernoulli-Euler
beam.

The problem of internal exact  controllability  was studied by Haraux {\bf\cite{HA1}}, Jaffard {\bf\cite{ja}} and Lions {\bf\cite{lio}}. The pointwise exact controllability for a strategic point was studied by Haraux and Jaffard   {\bf\cite{HA2}} and Lions {\bf\cite{lio}}. However, the convergence of the internal exact controllability of equation (\ref{p1})-(\ref{p3}) to the pointwise exact controllability has apparently not yet been studied.

The aim of this paper is to describe what happens when $n$ tends to
infinity, we can't hope to get a pointwise control for the limit
problem for any strategic  point in $(0,1)$,
 we use the same  techniques introduced in
{\bf\cite{FP1}}.

Our purpose in this paper is to  prove  the pointwise controllability of the
  Bernoulli-Euler beam equation. It is obtained as a limit of internal controllability of  the same type of equation.
Our approach based on the techniques used in {\bf\cite{FP1}}. This result can be proved by the standard HUM method
(Hilbert uniqueness method) by J.L. Lions {\bf\cite{lio}}. As n tends to infinity, we obtain the
solution of an exact internal controllability problem which converges towards to the
solution of an exact pointwise controllability  problem.

The plan of the paper is as follows. In section $2$ we show the regularity of weak solutions
of problem (\ref{p1})-(\ref{p3}) for a strategic point in $(0,1)$
and we  study the behavior of these solutions in an interval of length  $\frac{1}{n}$.
The  exact controllability results are given in section $3$ . In section $4$ we prove an
inverse inequality which will give us the estimates on the internal
controls in the case of  a strategic point. Finally, in section
$5$ as  n tends to infinity we prove that the pointwise
exact controllability problem is obtained as limit of exact internal controllability
problem of the beam equation.
 \section {Estimation and regularity results near a point}
Now introducing the Hilbert spaces
$$
V=\{u\in H^2(0,1),\,\,\, u(0)=0,\,\ \frac{d u}{d
x}(1)= 0\}.
$$
$V'$ is the dual space of $V$ with respect to the pivot space
$L^2(0,1)$, where the duality is in the sense of $L^2(0,1)$.\\
and
$${\mathcal D}(\partial_x^4) = \left\{u\in H^4(0,1),\,\,\, u(0)=\frac{d u}{d x}(1)=0,\,\,\,
 \frac{d^2 u}{d x^2}(0)=\frac{d^3 u}{d x^3}(1)=0\right\}.$$
 Consider two given functions $(u^0,u^1)$ in $
L^2(0,1)\times V'$,\,\ $g_n\in L^2(0,T,L^2(0,1))$ and
$\mathrm{supp}\,(g_n)=[\xi,\xi+\frac{1}{n}]$,\,\ $\xi \in (0,1)$ and
 we will take $\frac{1}{n} < 1-\xi$. Let
$u$ be the solution of (\ref{p1})-(\ref{p3}).
\begin{proposition}\label{pro}
Assume $g_n \in L^2(0,T,L^2(0,1))$ and $(u^0,u^1)$ in  $
L^2(0,1)\times V'$. Then for any $T> 0$, problem
(\ref{p1})-(\ref{p3}) admits a unique solution
$$
u  \in{\mathcal C}\left(0,T,L^2(0,1)\right) \cap {\mathcal
C}^1\left(0,T,V'\right).
$$
Moreover,
\begin{equation}\label{mi5}
u(\xi,t)\in L^2(0,T),
\end{equation}
and there exists a constant $C> 0$ (independent on $n$ and T), such that
\begin{equation}\label{mama}
n\Int_{\xi}^{\xi+\frac{1}{n}}\Int_0^T
\Big|u(x,t)\Big|^2\,dx\,dt
 \leq
C\Big(\|g_n\|^2_{L^{2}(0,T,L^2(0,1))} + \|u^0\|^2_{L^2(0,1)} +
\|u^1\|^2_{V'}\Big),
\end{equation}
for all $n>0$.
\end{proposition}
\begin{proof}
In order to prove (\ref{mama}) we put
$$
u^0(x)=\sum_{m=0}^\infty a_m\sin\Big(\frac{2m+1}{2}\pi x\Big),\,\,\,\,
u^1(x)=\sum_{m=0}^\infty \frac{b_m}{(\frac{2m+1}{2}\pi)^2}\sin\Big(\frac{2m+1}{2}\pi x\Big),
$$
and
$$
g_n(x,t)=\sum_{m=0}^\infty g_m(t)\sin\Big(\frac{2m+1}{2}\pi x\Big),
$$
with $(a_m), \Big(\frac{b_m}{(\frac{2m+1}{2}\pi)^2}\Big)\in l^2(\R)$ and for $t$ fixed $(g_m(t))\in l^2(\R)$.\\
\\
\\
The solution of (\ref{p1})-(\ref{p3}) is given by
{\small\begin{equation}\label{CV}
\begin{split}
u(x,t)&=\sum_{m=0}^\infty\left\{ a_m\cos\Big[\Big(\frac{2m+1}{2}\pi\Big)^2t\Big] +
\frac{b_m}{(\frac{2m+1}{2}\pi)^4}\sin\Big[\Big(\frac{2m+1}{2}\pi\Big)^2t\Big]\right.\\&+\left.
\frac{1}{(\frac{2m+1}{2}\pi)^4}\Int_0^t
\sin\Big[\Big(\frac{2m+1}{2}\pi\Big)^2(t-s)\Big]g_m(s)\,ds\right\}\sin\Big(\frac{2m+1}{2}\pi x\Big).
\end{split}
\end{equation}}
Which implies that
{\small\begin{equation*}
\begin{split}
u(\xi,t)&=\sum_{m=0}^\infty\left\{ a_m\cos\Big[\Big(\frac{2m+1}{2}\pi\Big)^2t\Big] +
\frac{b_m}{(\frac{2m+1}{2}\pi)^4}\sin\Big[\Big(\frac{2m+1}{2}\pi\Big)^2t\Big]\right.\\&+\left.
\frac{1}{(\frac{2m+1}{2}\pi)^4}\Int_0^t
\sin\Big[\Big(\frac{2m+1}{2}\pi\Big)^2(t-s)\Big]g_m(s)\,ds\right\}\sin\Big(\frac{2m+1}{2}\pi \xi\Big).
\end{split}
\end{equation*}}
We see that
$$
\|u^0\|^2_{L^2(0,1)}+ \|u^1\|^2_{V'}= \frac{1}{2}\sum_{m=0}^\infty\Big[a_m^2+
\frac{b_m^2}{(\frac{2m+1}{2}\pi)^8}\Big],
$$
and
$$
\|g_n\|^2_{L^2(0,1)}=\frac{1}{2}\sum_{m=0}^\infty g^2_m(t).
$$
Integrating (\ref{CV}) over $(0,1)$, we get
{\small\begin{equation}\label{BBN}
\begin{split}
\Int_0^1u^2(x,t)\,dx&=\Frac{1}{2}\sum_{m=0}^\infty\left\{a_m^2\cos^2\Big[\Big(\frac{2m+1}{2}\pi\Big)^2t\Big]+
\frac{b^2_m}{(\frac{2m+1}{2}\pi)^8}\sin^2\Big[\Big(\frac{2m+1}{2}\pi\Big)^2t\Big]\right.\\&+\left.
\frac{1}{(\frac{2m+1}{2}\pi)^8}\Big(\Int_0^t
\sin\Big[\Big(\frac{2m+1}{2}\pi\Big)^2(t-s)\Big]g_m(s)\,ds\Big)^2\right.\\&+\left.2a_m \frac{b_m}{(\frac{2m+1}{2}\pi)^4}\cos\Big[\Big(\frac{2m+1}{2}\pi\Big)^2t\Big]\sin \Big[\Big(\frac{2m+1}{2}\pi\Big)^2t\Big]
\right.\\&+\left.2\frac{b_m}{(\frac{2m+1}{2}\pi)^8}\sin \Big[\Big(\frac{2m+1}{2}\pi\Big)^2t\Big]\Int_0^t
\sin\Big[\Big(\frac{2m+1}{2}\pi\Big)^2(t-s)\Big]g_m(s)\,ds
\right.\\&+\left.\frac{2a_m}{(\frac{2m+1}{2}\pi)^4} \cos \Big[\Big(\frac{2m+1}{2}\pi\Big)^2t\Big]\Int_0^t
\sin\Big[\Big(\frac{2m+1}{2}\pi\Big)^2(t-s)\Big]g_m(s)\,ds
\right\}
\end{split}
\end{equation}}
we shall estimate the third term of right hand side of (\ref{BBN}), we used H\"{o}lder inequality, we get
{\small\begin{equation}\label{BBN1}
\begin{split}
\sum_{m=0}^\infty \Frac{1}{(\frac{2m+1}{2}\pi)^8}\Big(\Int_0^t
\sin\Big[\Big(\frac{2m+1}{2}\pi\Big)^2(t-s)\Big]g_m(s)\,ds\Big)^2
&\leq \sum_{m=0}^\infty \Int_0^t \sin^2\Big[\Big(\frac{2m+1}{2}\pi\Big)^2(t-s)\Big]\,ds \Int_0^t g^2_m(s)\,ds
\\&\leq C(T)\sum_{m=0}^\infty \Int_0^Tg^2_m(t)\,dt.
\end{split}
\end{equation}}
By Young's inequality, we get
{\small\begin{equation}\label{BBN2}
\begin{split}
\sum_{m=0}^\infty
2a_m \frac{b_m}{(\frac{2m+1}{2}\pi)^4}\cos\Big[\Big(\frac{2m+1}{2}\pi\Big)^2t\Big]\sin \Big[\Big(\frac{2m+1}{2}\pi\Big)^2t\Big]
\leq c_1\sum_{m=0}^\infty a_m^2 +  c'_1\sum_{m=0}^\infty \frac{b^2_m}{(\frac{2m+1}{2}\pi)^8},
\end{split}
\end{equation}}
and  using Young's and H\"{o}lder inequalities, we have
{\small\begin{equation}\label{BBN3}
\begin{split}
\sum_{m=0}^\infty 2\frac{b_m}{(\frac{2m+1}{2}\pi)^8}\sin \Big[\Big(\frac{2m+1}{2}\pi\Big)^2t\Big]&
\Int_0^t\sin\Big[\Big(\frac{2m+1}{2}\pi\Big)^2(t-s)\Big]g_m(s)\,ds
\\&\leq C(T)\sum_{m=0}^\infty \frac{b_m}{(\frac{2m+1}{2}\pi)^8}\Big(\Int_0^T g^2_m(t)\,dt\Big)^{\frac{1}{2}}
\\&\leq
 c_2 \sum_{m=0}^\infty\frac{b^2_m}{(\frac{2m+1}{2}\pi)^{16}} +
  c'_2 \sum_{m=0}^\infty \Int_0^T g^2_m(t)\,dt
  \\& \leq
 c_2 \sum_{m=0}^\infty\frac{b^2_m}{(\frac{2m+1}{2}\pi)^8} +
  C(T)\sum_{m=0}^\infty \Int_0^T g^2_m(t)\,dt,
\end{split}
\end{equation}}
and
{\small\begin{equation}\label{BBN4}
\begin{split}
\sum_{m=0}^\infty \frac{2a_m}{(\frac{2m+1}{2}\pi)^4}\cos \Big[\Big(\frac{2m+1}{2}\pi\Big)^2t\Big]&\Int_0^t
\sin\Big[\Big(\frac{2m+1}{2}\pi\Big)^2(t-s)\Big]g_m(s)\,ds\\&\leq
 c_3 \sum_{m=0}^\infty a_m^2 +
  C(T)\sum_{m=0}^\infty \Int_0^T g^2_m(t)\,dt.
\end{split}
\end{equation}}
Integrating (\ref{BBN}) in $(0,T)$ and using (\ref{BBN1})-(\ref{BBN4}), we obtain from (\ref{BBN}) that
{\small\begin{equation*}
\begin{split}
\Int_0^1 \Int_0^Tu^2(x,t)\,dt\,dx&\leq \frac{C(T)}{2}\sum_{m=0}^\infty\left\{a_m^2+\frac{b_m^2}{(\frac{2m+1}{2}\pi)^8}
+\Int_0^Tg^2_m(t)\,dt\right\}.
\end{split}
\end{equation*}}
Then
{\small\begin{equation*}
\begin{split}
\Int_\xi^{\xi+\frac{1}{n}}\Int_0^Tu^2(x,t)\,dt\,dx \leq \frac{C(T)}{n}\Big(\|u^0\|^2_{L^2(0,1)} + \|u^1\|^2_{V'}
+\|g_n\|^2_{L^2(0,T,L^2(0,1))}\Big).
\end{split}
\end{equation*}}
This completes the proof of proposition \ref{pro}.\end{proof}
\section{Internal exact controllability of the beams
equation} We consider now the following homogenous problem
\begin{equation}\label{p4o}
\Frac{\partial ^2 \phi}{\partial
t^2}(x,t)+\Frac{\partial^{4}\phi}{\partial x^4}(x,t)= 0,\,\,\,
0<x<1,\,\,\,\ t>0,
\end{equation}
\begin{equation}\label{p5o}
\phi(0,t)= \Frac{\partial \phi}{\partial x}(1,t)= \Frac{\partial^2
\phi}{\partial x^2}(0,t)= \Frac{\partial^3 \phi}{\partial
x^3}(1,t)=0,
\end{equation}
\begin{equation}\label{p6o}
\phi(x,0)=\phi^{0}(x),\,\,\,\,\ \Frac{\partial \phi}{\partial
t}(x,0)=\phi^{1}(x),\,\,\ 0<x<1.\end{equation} where $(\phi^0,
\phi^1) \in L^2(0,1) \times V'$.
\begin{lemma}\label{lem2}
Let $\xi \in (0,1)$, then for any natural integer $m$ we have
{\small\begin{equation}\label{tu111}
\begin{split} \Inf_{m\geq 0} \Int_{\xi}^{\xi
+\frac{1}{n}}\sin^2\left(\frac{2m+1}{2}\pi x\right)\,dx &\geq
\frac{1}{2n} - \Frac{1}{\pi}\sin\Big( \frac{\pi}{2n}\Big)
\\& \geq c_{\pi}o\Big(\frac{1}{n^3}\Big).
\end{split}
\end{equation}}
\end{lemma}
\begin{proof} For $m \geq 0$, it is sufficient to note that for
any $x\in(0,1)$, we have {\small\begin{equation*}
\begin{split} \Int_{\xi}^{\xi
+\frac{1}{n}}\sin^2\left(\frac{2m+1}{2}\pi x\right)\,dx &=
\frac{1}{2m+1}\Sum_{k=0}^{2m}\Int_{(2m+1)\xi+\frac{k}{n}}^{(2m+1)\xi+\frac{k+1}{n}}\sin^2
\Big(\frac{\pi}{2}y\Big)\,dy \\& = \frac{1}{2m+1}\Int_{(2m+1)\xi}
^{(2m+1)(\xi+\frac{1}{n})}\sin^2\Big(\frac{\pi}{2}y\Big)\,dy\\& = \frac{1}{2m+1}\Int_{(2m+1)\xi}
^{(2m+1)(\xi+\frac{1}{n})}\Big(\Frac{1}{2}-\Frac{1}{2}\cos(\pi y)\Big)\,dy
\\& = \frac{1}{2n}-
\frac{1}{\pi(2m+1)}\sin\Big[\frac{\pi}{2n}(2m+1)\Big]\cos\Big[\Big(\xi
+\frac{1}{2n}\Big) \Big(2m+1\Big)\Big],
\end{split}
\end{equation*}}
it is clear that
$$
\sin\Big( \frac{\pi}{2n}(2m+1)\Big)= \frac{\pi}{2n}(2m+1) +c_{\pi,m}
o\Big(\frac{1}{n^3}\Big).
$$
Therefore, we get {\small\begin{equation*}
\begin{split}\Int_{\xi}^{\xi
+\frac{1}{n}}\sin^2\left(\frac{2m+1}{2}\pi x\right)\,dx & \geq
\frac{1}{2n} - \Frac{1}{\pi}\Big[\frac{\pi}{2n}
+o\Big(\frac{1}{n^3}\Big)\Big] \\& \geq
c_{\pi}o\Big(\frac{1}{n^3}\Big)
 .\end{split}
\end{equation*}}
The proof of lemma \ref{lem2} is now completed.
\end{proof}
The previous lemma is an essential tool  to show the following
Proposition.
\begin{proposition}\label{pro2}
Let $T\geq 2$, then we have the following. For almost all $\xi \in
(0,1)$ the solution $\phi$ of (\ref{p4o})-(\ref{p6o}) satisfies
\begin{equation}\label{obser1}
\Int_0^T\Int_\xi ^{\xi +\frac{1}{n}}\Big|\phi(x,t)
\Big|^2\,dx\,dt\geq
 c_{\pi,n}(\|u^0\|^2_{L^2(0,1)} + \|u^1\|^2_{V'}),
\end{equation}
$$
\forall (u^0,u^1) \in  L^2(0,1)\times V',
$$
where $c_{\pi,n} = c_{\pi}o(\frac{1}{n^3})$.\\
\end{proposition}
\begin{proof}
 The solution of (\ref{p4o})-(\ref{p6o})  is given by
$$
\phi(x,t)=
\sum_{m=0}^{\infty}\left\{a_m\cos\Big[\left(\frac{2m+1}{2}\pi\right)^2
t\Big]+
\frac{b_m}{\left(\frac{2m+1}{2}\pi\right)^4}\sin\Big[\left(\frac{2m+1}{2}\pi\right)^2
t\Big]\right\}\sin\left(\frac{2m+1}{2}\pi x\right).
$$
A simple calculation shows that
 {\small\begin{equation}\label{amma}
 \begin{split} \Int_0^2\Big|\phi(x,t)\Big|^2\,dt
 =\Sum_{m=0}^{\infty}\left\{a_m^2+\frac{b_m^2}{\left(\frac{2m+1}{2}\pi\right)^8}\right\}
 \sin^2\left(\frac{2m+1}{2}\pi
x\right),
\end{split}
\end{equation}}
we have {\small\begin{equation*}
\begin{split}\Int_{\xi}^{\xi +\frac{1}{n}} \Int_0^2\Big|\phi(x,t)\Big|^2\,dx
\,dt \geq
\Sum_{m=0}^{\infty}\left\{a_m^2+\frac{b_m^2}{\left(\frac{2m+1}{2}\pi\right)^8}\right\}\Inf_{m\geq 0} \Int_{\xi}^{\xi
+\frac{1}{n}}\sin^2\left(\frac{2m+1}{2}\pi x\right)\,dx,
\end{split}
\end{equation*}}
for every $T \geq 2$, we get
\begin{equation}\label{tu11}
\Int_{\xi}^{\xi +\frac{1}{n}} \Int_0^T\Big|\phi(x,t)\Big|^2\,dx \,dt
\geq
\Sum_{m=0}^{\infty}\left\{a_m^2+\left(\frac{b_m}{\left(\frac{2m+1}{2}\pi\right)^4}\right)^2\right\}\Inf_{m\geq
0} \Int_{\xi}^{\xi +\frac{1}{n}}\sin^2\left(\frac{2m+1}{2}\pi
x\right)\,dx.
\end{equation}
 Consequently, by  lemma \ref{lem2} and  using (\ref{tu111}) and
(\ref{tu11}), we obtain (\ref{obser1}).\\
 This achieve the proof of proposition \ref{pro2}.
\end{proof}
Let $(y^0,y^1)\in
 V \times L^2(0,1)$ and  $\psi_n$ be the solution of
\begin{equation}\label{23}
\Frac{\partial ^2 \psi_n}{\partial
t^2}(x,t)+\Frac{\partial^{4}\psi_n}{\partial x^4}(x,t)=
\chi_n(x)\widetilde{\phi}_n(x,t) ,\,\,\,\ 0<x<1,\,\,\,\,\ t>0,
\end{equation}
\begin{equation}\label{p5oo}
\psi_n(0,t)= \Frac{\partial \psi_n}{\partial x}(1,t)= \Frac{\partial^2
\psi_n}{\partial x^2}(0,t)= \Frac{\partial^3 \psi_n}{\partial x^3}(1,t)=0,
\end{equation}
\begin{equation}\label{p6oo}
\psi_n(x,0)=y^{0}(x),\,\,\,\,\ \Frac{\partial \psi_n}{\partial
t}(x,0)=y^{1}(x),\,\,\ 0<x<1,
\end{equation}
\begin{equation}\label{p7oo}
\psi_n(x,T)=\frac{\partial\psi_n}{\partial t}(x,T)=0,
\end{equation}
where $\chi_n$ is the characteristic function of
$(\xi,\xi+\frac{1}{n})$ and  $\widetilde{\phi}_n(x,t)= n\phi(x,t)$,
$\phi$ is the solution of (\ref{p4o})-(\ref{p6o}).
\begin{lemma}{\bf\cite{HA1}}\label{mo}
We suppose that $(y^0,y^1) \in V \times L^2(0,1)$, then
$$
\Big(y^0,
\widetilde{\phi}_n^1\Big)-\Big(y^1,\widetilde{\phi}_n^0\Big)=\Int_\xi^{\xi+\frac{1}{n}}\Int_0^T
\Big|\widetilde{\phi}_n(x,t)\Big|^2\,dx\,dt.
$$
\end{lemma}
\section{An inverse inequality}
 In this section we suppose  that the point  $\xi$ is strategic, that's
\begin{equation}\label{H}
  \sin\Big(\frac{2m+1}{2}\pi \xi\Big)\neq 0,\quad\forall\,m\in\mathbb{N}.
  \end{equation}
   The quantity
$$
\Big(\Int_0^T \phi^2(\xi,t)\,dt\Big)^{(1/2)}
$$
where $\phi$ is a solution of (\ref{p4o})-(\ref{p6o}) defines a norm
on the space ${\mathcal D}(0,1)\times {\mathcal D}(0,1)$ and the
initial data $\phi^0$ and $\phi^1$  are given by
$$
\phi^0= \Sum_{m=0}^{\infty} a_m \sin\Big(\frac{2m+1}{2}\pi
x\Big),\,\,\,\,\, \phi^1= \Sum_{m=0}^{\infty}
\frac{b_m}{(\frac{2m+1}{2}\pi)^2}\sin\Big(\frac{2m+1}{2}\pi x\Big).
$$
Let $F$ be a real Hilbert space
$$
(\phi^0,\phi^1)\in F\Leftrightarrow
\Sum_{m=0}^{\infty}\Big\{a_m^2+\frac{b_m^2}{(\frac{2m+1}{2}\pi)^8}\Big\}\sin^2\Big(\frac{2m+1}{2}\pi
\xi\Big) <\infty.
$$
We denote by $F$ the completion of ${\mathcal D}(0,1)\times {\mathcal D}(0,1)$ for this norm and we denote
 by $\|.\|_F$ the following quantity:
$$
\|\phi\|_{F}=\Big(\Int_0^T \phi^2(\xi,t)\,dt\Big)^{1/2}.
$$
Therefore
$$
L^2(0,1)\times V'\subset F.
$$
If
$$
y^0(x)=\Sum_{m=0}^{\infty}a_m\sin\Big(\frac{2m+1}{2}\pi x\Big),\,\,\,\,\,\,\,\,\
y^1(x)=\Sum_{m=0}^{\infty}\frac{b_m}{(\frac{2m+1}{2}\pi)^2}\sin\Big(\frac{2m+1}{2}\pi x\Big),
$$
 and therefore, its dual
$$
(y^0,y^1)\in F'\Leftrightarrow \Sum_{m=0}^{\infty}\frac{a_m^2+\frac{b_m^2}{(\frac{2m+1}{2}\pi)^4}}{\sin^2\Big(\frac{2m+1}{2}\pi \xi\Big)}
<\infty.
$$
\begin{remark}
If $\xi\in (0,1)$ satisfying (\ref{H}),  then there exists a
constant $C> 0$ such that $|\sin\Big(\frac{2m+1}{2}\pi\xi\Big)|\geq
C,\,\,\, \forall m\in \N$
 and therefore  $L^2(0,1)\times V'= F$.\\
  For the proof, see {\bf\cite{AmTu2}},{\bf\cite{AmTu1}}.
\end{remark}
 The main result of this section is the following:
\begin{theorem}
For $T\geq 2$, there exists $c> 0$, such that for $(\phi^0,\phi^1)
\in F$, the solution $\phi$ of (\ref{p4o})-(\ref{p6o}) satisfies
\begin{equation}\label{ZZ}
\|(\phi^0,\phi^1)\|_F^2 \leq c \Big(n
\Int_0^T\Int_\xi^{\xi+\frac{1}{n}} \phi^2(x,t)\,dx\,dt\Big).
\end{equation}
\end{theorem}
\begin{proof}
For $T\geq 2$. Using (\ref{amma}), we have
{\small\begin{equation*}
\begin{split}
\Int_0^T\Int_\xi^{\xi+\frac{1}{n}}\phi^2(x,t)\,dx\,dt&\geq
\Int_0^2\Int_\xi^{\xi+\frac{1}{n}}\phi^2(x,t)\,dx\,dt\\&=
\sum_{m=0}^{\infty}\left\{a_m^2+\left(\frac{b_m}{\left(\frac{2m+1}{2}\pi\right)^4}\right)^2\right\}
\Int_\xi^{\xi+\frac{1}{n}}\sin^2\left(\frac{2m+1}{2}\pi
x\right)\,dx.
\end{split}
\end{equation*}}
Now, we have to prove that there exists $c> 0$ independent on $n$
such that for every integer $m\in \N$, we have
$$
n\Int_\xi^{\xi+\frac{1}{n}}\sin^2\left(\frac{2m+1}{2}\pi x\right)\,dx\geq
 c \sin^2\left(\frac{2m+1}{2}\pi \xi\right).
$$
For $b\geq 0,\,\, t\geq 0$, we set
$$
I(b,t)= \Int_0^1\sin^2(\pi(b+tz))\,dz.
$$
As
$$
n\Int_\xi^{\xi+\frac{1}{n}}\sin^2\left(\frac{2m+1}{2}\pi x\right)\,dx =
 I\Big(\frac{2m+1}{2}\xi,\frac{2m+1}{2n}\Big),
$$
it is sufficient to prove that there exists $c> 0$ such that
\begin{equation}\label{tt}
\forall t\geq 0,\,\, I(b,t) \geq c\sin^2(\pi b).
\end{equation}
We have the formula
{\small\begin{equation*}
\begin{split}
\forall t\geq 0,\,\, I(b,t)&= \Frac{1}{2}\Big(1-\Frac{\sin (2\pi (b+t))-\sin (2\pi b)}{2\pi t}\Big)\\&=
\Frac{1}{2}\Big(1-\Frac{\sin (2\pi b)[\cos (2\pi t)-1] +\sin (2\pi t)\cos (2\pi b)}{2\pi t}\Big)
\\&=
\Frac{1}{2}\Big(1-\Frac{-2\sin (2\pi b)\sin^2(\pi t) +2\cos (\pi t)\sin (\pi t)\cos (2\pi b)}{2\pi t}\Big)\\&=
\frac{1}{2}\left(1-\frac{\sin (\pi t)}{\pi t}\cos(\pi(2b+t))\right)
\end{split}
\end{equation*}}
If $t \geq \frac{1}{2}$, then $I(b,t)\geq \frac{1}{2}(1-\frac{2}{\pi})$\\
If $t< \frac{1}{2}$, we distinguish two cases:\\
{\bf{Case 1.}} $]b,b+t[\subset ]p,(p+1)[$. It is then enough to
consider the case $p=0$ and as $\sin(\pi .)$ is concave on $[0,1]$, we
obtain
{\small\begin{equation*}
\begin{split}
\forall z \in (0,1),\,\,\, \sin ((1-z)b\pi + (b+t)z\pi)&=\sin ((1-z)\pi b+ z\pi(t+b))\\&\geq (1-z)\sin(\pi b)+z\sin (\pi(t+b))\\&\geq
(1-z)\sin(\pi b).
\end{split}
\end{equation*}}
Then
$$
\forall z \in (0,1),\,\,\, |\sin(\pi( b+tz))|\geq (1-z)|\sin (\pi b)|.
$$
Hence
\begin{equation*}
 \begin{split}
I(b,t)&\geq \sin^2 (\pi b)\Int_0^1(1-z)^2\,dz\\&
\geq \frac{1}{3}\sin^2 (\pi b).
\end{split}
\end{equation*}
{\bf{Case 2.}} $p-1\leq b\leq p\leq b+t \leq p+1$. It is enough
here to consider the case $p=1$ and we write $1= b+z_0t$ writh
$z_0 \in (0,1)$. We have
$$
|\sin (\pi (b+tz))|\geq (1-z)|\sin (\pi b)|\,\,\,\,\mbox{for}\,\,\,\, z\leq z_0,
$$
and
$$
|\sin (\pi (b+tz))|\geq z|\sin (\pi( b+t))|\,\,\,\,\mbox{for}\,\,\,\, z>z_0.
$$
Now, if $z_0 \geq \frac{1}{2}$, we find
\begin{equation*}
 \begin{split}
I(b,t)&\geq \sin^2 (\pi b) \Int_0^{z_0} (1-z)^2\,dz+\sin^2 (\pi b) \Int_{z_0}^1 (1-z)^2\,dz\\& \geq \sin^2 (\pi b) \Int_0^{z_0} (1-z)^2\,dz \\& \geq \frac{7}{24}\sin^2(\pi b).
\end{split}
\end{equation*}
If $z_0< \frac{1}{2}$, then
{\small\begin{equation*}
\begin{split}b+t-1= b+t-(b+z_0t)& = 1-2z_0t +t -b\\&= 1-b +t(1-2z_0)\\&> 1-b
\end{split}
\end{equation*}}
and
$$
\sin(\pi(b+t-1))= -\sin (\pi (b+t))
$$
 and we have
\begin{equation*}
 \begin{split}
I(b,t)&\geq\sin^2(\pi(b+t-1))\Int_0^{z_0} z^2\,dz
+\sin^2(\pi(b+t-1))\Int_{z_0}^1 z^2\,dz
\\&\geq \sin^2(\pi(b+t-1))\Int_{z_0}^1 z^2\,dz
 \\&\geq  \frac{1}{3}(1-z_0^3)\sin^2(\pi(1-b))\\&\geq \frac{7}{24}\sin^2(\pi b).
\end{split}
\end{equation*}
The proof of the theorem \ref{pro2} is complete.
\end{proof}
\section{Estimates on the controls}
For $T\geq 2$ and $\frac{1}{n}\widetilde{\phi}_n(x,t) =\phi(x,t)$
where $\phi$ is the solution of (\ref{p4o})-(\ref{p6o}), we have
\begin{theorem}\label{RR}
1. If $(y^0,y^1)\in V\times L^2(0,1)$, we have
\begin{equation}\label{a}
\|\widetilde{\phi}_n^0\|_{L^2(0,1)}+\|\widetilde{\phi}_n^1\|_{V'}=o(n^3),
\end{equation}
and
\begin{equation}\label{aa}
\Int_\xi^{\xi+\frac{1}{n}}\Int_0^T\Big
|\widetilde{\phi}_n(x,t)\Big|^2\,dx\,dt=o(n^3).
\end{equation}
2. If $\xi$ is strategic and $(y^0,y^1)\in F'$, we have
\begin{equation}\label{aaa}
\|(\widetilde{\phi}_n^0,\widetilde{\phi}_n^1)\|_{F}=o(n),
\end{equation}
and
\begin{equation}\label{aaaa}
\Int_\xi^{\xi+\frac{1}{n}}\Int_0^T\Big
|\widetilde{\phi}_n(x,t)\Big|^2\,dx\,dt=o(n).
\end{equation}
\end{theorem}
\begin{proof} 1. Applying H\"{o}lder and Young's inequalities. Hence, we see from (\ref{obser1}) and using lemma \ref{mo}, we have
 {\small\begin{equation}\label{oo}
 \begin{split}
\Big(\|\widetilde{\phi}_n^0\|_{L^2(0,1)}+\|\widetilde{\phi}_n^1\|_{V'}\Big)^2&\leq
c\Big(\|\widetilde{\phi}_n^0\|^2_{L^2(0,1)}+\|\widetilde{\phi}_n^1\|^2_{V'}\Big)\\&\leq
c n^2\Big(\|\phi^0\|^2_{L^2(0,1)}+\|\phi^1\|^2_{V'}\Big)
\\&\leq
c n^3 \Int_\xi^{\xi+\frac{1}{n}}\Int_0^T\Big
|\widetilde{\phi}_n(x,t)\Big|^2\,dx\,dt
\\& \leq c n^3
(\|\widetilde{\phi}_n^0\|_{L^2(0,1)}\|y^1\|_{L^2(0,1)}+\|\widetilde{\phi}_n^1\|_{V'}\|y^0\|_{V})
\\& \leq c n^3
(\|\widetilde{\phi}_n^0\|_{L^2(0,1)}+\|\widetilde{\phi}_n^1\|_{V'})
\end{split}
\end{equation}}
2. When the point $\xi$ is strategic and the initial data $(y^0,y^1)\in F'$.
Hence, we see from (\ref{ZZ}), that
 {\small\begin{equation*}
 \begin{split}
\|(\widetilde{\phi}_n^0,\widetilde{\phi}_n^1)\|^2_{F}&\leq c n
\Int_\xi^{\xi+\frac{1}{n}}\Int_0^T\Big
|\widetilde{\phi}_n(x,t)\Big|^2\,dx\,dt\\& \leq c n
\|(\widetilde{\phi}_n^0,
\widetilde{\phi}_n^1)\|_F\|(y^0,y^1)\|_{F'}\\& \leq c n
\|(\widetilde{\phi}_n^0, \widetilde{\phi}_n^1)\|_F.
\end{split}
\end{equation*}}
The proof of Theorem \ref{RR} is now complete.
\end{proof}
\section{Controllability limit as $n\rightarrow \infty$}
We study here the possibility of convergence of the solution of the
controllability  problems defined by (\ref{p5oo})-(\ref{p7oo}).
This convergence depends on the nature of the point $\xi$ and on
the space of the initial data $y^0$ and $y^1$.\\
If the point (\ref{H}) holds, we have the following theorem.
\begin{theorem}
Suppose that $T\geq 2$, if $\xi$ checks (\ref{H}), $y^0$ and $y^1$
belong to $F'$.\\
 Then, the solution of (\ref{p5oo})-(\ref{p7oo})
converges for the weak* topology of $L^{\infty}(0,T,V)$ to the
solution of the following pointwise system
\begin{equation}\label{p44oo}
\Frac{\partial ^2 \psi}{\partial
t^2}(x,t)+\Frac{\partial^{4}\psi}{\partial x^4}(x,t)=
v(t)\delta_\xi,\,\,\, 0<x<1,\,\,\,\ t>0,
\end{equation}
\begin{equation}\label{p55oo}
\psi(0,t)= \Frac{\partial \psi}{\partial x}(1,t)= \Frac{\partial^2
\psi}{\partial x^2}(0,t)= \Frac{\partial^3 \psi}{\partial x^3}(1,t)=0,
\end{equation}
\begin{equation}\label{p66oo}
\psi(x,0)=y^{0}(x),\,\,\,\,\ \Frac{\partial \psi}{\partial
t}(x,0)=y^{1}(x),\,\,\ 0<x<1,\end{equation}
\begin{equation}\label{p77oo}
\psi(x,T)=\frac{\partial \psi}{\partial t}(x,T)=0,\,\,\ 0<x<1,
\end{equation}
where $v\in L^2(0,T)$ and $ \phi(\xi,t)+\frac{1}{2n} \frac{\partial
\phi}{\partial x}(\xi,t)$ converges for the weak* topology of to
$v(t)$ in $H^{-1}(0,T)$.
\end{theorem}
\begin{proof} Multiplying (\ref{p1}) by $\psi_n(x,t)$ and
integrating by parts on  $(0,T)\times(0,1)$, we have\\

$
\forall (u^0,u^1,g_n) \in L^2(0,1)\times V'\times L^2(0,T,L^2(0,1)),
$
{\small\begin{equation}\label{AQZ}
 \begin{split}
&\Int_0^T\Int_0^{1}g_n(x,t)\psi_n(x,t)\,dx\,dt\\&=
\Int_0^T\Int_\xi^{\xi+\frac{1}{n}}\widetilde{\phi}_n(x,t)u(x,t)\,dx\,dt-
\Int_0^1y^0(x)u^1(x)\,dx +\Int_0^1 y^1(x)u^0(x)\,dx
\\&=n
\Int_0^T\Int_\xi^{\xi+\frac{1}{n}}\phi(x,t)u(x,t)\,dx\,dt-
\Int_0^1y^0(x)u^1(x)\,dx +\Int_0^1 y^1(x)u^0(x)\,dx.
\end{split}
\end{equation}}
Now, we prove that $(\psi_n)$  and $(g_n)$ are  bounded in
$L^{\infty}(0,T,L^2(0,1))$.\\Define
$$
\begin{array}{ll}
K_n: L^2(0,1)\times V'\times L^2(0,T,L^2(0,1))\rightarrow \R\\
\hskip4.1cm (u^0,u^1,g_n)\mapsto n
\Int_0^T\Int_\xi^{\xi+\frac{1}{n}}\phi(x,t)u(x,t)\,dx\,dt.
\end{array}
$$
Using H\"{o}lder inequality, we have
\begin{equation}\label{A}
|K_n(u^0,u^1,g_n)|^2 \leq \Big(n
\Int_0^T\Int_\xi^{\xi+\frac{1}{n}}\Big|\phi(x,t)\Big|^2\,dx\,dt\Big)\Big(
n\Int_0^T\Int_\xi^{\xi+\frac{1}{n}}|u(x,t)|^2\,dx\,dt\Big).
\end{equation}
Replacing (\ref{mama}) in (\ref{A}) and from (\ref{aaaa}), we have
$$
|K_n(u^0,u^1,g_n)|^2 \leq c\Big(\|u^0\|^2_{L^2(0,1)}+\|u^1\|^2_{V'} + \|g_n\|^2_{L^{2}(0,T,L^2(0,1))}
 \Big),
$$
which proves that the linear forms $K_n$ are bounded in $L^2(0,1) \times V' \times L^\infty(0,T,L^2(0,1))$.\\
Therefore, $(\psi_n)_n$ and $(g_n)_n$ are bounded in $L^\infty(0,T,L^2(0,1))$
after extraction of a subsequence of $(\psi_n)_n$ and $(g_n)_n$ still denoted by
$(\psi_n)_n$ and $(g_n)_n$, such that
$$
\psi_n \rightharpoonup \psi\,\,\,\mbox{ weakly* in }\,\,\,L^\infty(0,T,L^2(0,1)),
$$
and
$$
g_n \rightharpoonup g\,\,\,\mbox{ weakly* in }\,\,\,L^\infty(0,T,L^2(0,1)).
$$
The limit of $K_n$ is given  in the following lemma.
\begin{lemma}\label{50}
 The linear forms $K_n$ converge in $ L^2(0,1)\times V \times
L^\infty(0,T,L^2(0,1))$ weakly* to the $K$ defined by
\begin{equation}\label{sv}
K(u^0,u^1,g)=\Int_0^T v(t)u(\xi,t)\,dt,
\end{equation}
where $v\in L^2(0,T)$ and
$$
 \phi(\xi,t)+\frac{1}{2n}
\Frac{\partial \phi}{\partial x}(\xi,t)\rightharpoonup
 v(t)\,\,\,\mbox{weakly* in}\,\,\,\
 H^{-1}(0,T).
$$
\end{lemma}
In order to prove the previous lemma, we  need  the following
result.
\begin{lemma}\label{mm1}
Let $(\phi^0,\phi^1)\in L^2(0,1)\times V'$ and the solution
$\phi(x,t)$ of the problem (\ref{p4o})-(\ref{p6o}) satisfies
\begin{equation}\label{R}
\Int_0^T\Int_\xi^{\xi+\frac{1}{n}} |\phi(x,t)|^2\,dx\,dt =
o\Big(\frac{1}{n}\Big).
\end{equation}
Then, after extraction of a subsequence
$$
\phi(\xi,t)+\frac{1}{2n} \frac{\partial \phi}{\partial x}(\xi,t)
\rightharpoonup  v(t)\,\,\,\mbox{weakly* in}\,\,\,\
  H^{-1}(0,T),
$$ where  $v\in L^{2}(0,T)$.
\end{lemma}
{\bf{Proof of lemma \ref{mm1}.}} In order to prove  lemma \ref{mm1}
we suppose that $w = \Big(\frac{\partial^4}{\partial
x^4}\Big)^{-1}u$ \\
such that $w$ is the solution of
\begin{equation}\label{pp1}
\Frac{\partial ^2 w}{\partial t^2}(x,t)+\Frac{\partial^{4}w}{\partial x^4}(x,t)=
f_n(x,t),\,\,\, 0<x<1,\,\,\,\ t>0,
\end{equation}
\begin{equation}\label{pp2}
w(0,t)= \Frac{\partial w}{\partial x}(1,t)= \Frac{\partial^2
w}{\partial x^2}(0,t)= \Frac{\partial^3 w}{\partial x^3}(1,t)=0,
\end{equation}
\begin{equation}\label{pp3}
w(x,0)=w^{0}(x),\,\,\,\,\ \Frac{\partial w}{\partial
t}(x,0)=w^{1}(x),\,\,\ 0<x<1.
\end{equation}
with initial data
\begin{equation}\label{ggg}
 \left\{
\begin{array}{ll}
w^0=\Big(\frac{\partial^4}{\partial x^4}\Big)^{-1}u^0
\in {\mathcal D}(\partial_x^4)\\
w^1=\Big(\frac{\partial^4}{\partial x^4}\Big)^{-1}u^1\in V\\
f_n=\Big(\frac{\partial^4}{\partial x^4}\Big)^{-1}g_n \in L^2(0,T;{\mathcal D}(\partial_x^4)).
\end{array} \right.
\end{equation}
The trace regularity for (\ref{pp1})-(\ref{pp3}) is given in the
theorem below.
\begin{theorem}\label{h}
 Suppose that $f_n\in L^2(0,T;{\mathcal D}(\partial_x^4))$ and
$(w^0,w^1)\in {\mathcal D}(\partial_x^4)\times V$ the solution $w$
of (\ref{pp1})-(\ref{pp3})
 verifies
\begin{equation}\label{gg}
\Frac{\partial^4 w}{\partial x^4}(\xi,t)\in L^2(0,T),
\end{equation}
and the mapping
\begin{equation}\label{ggg}
\begin{array}{ll}
L^2(0,T;{\mathcal D}(\partial_x^4))\times {\mathcal
D}(\partial_x^4)\times V\rightarrow L^2(0,T)\\ \hskip2.75cm
(f_n,w^0,w^1)\mapsto \Frac{\partial^4 w}{\partial x^4}(\xi,t),
\end{array}
\end{equation}
is linear and continuous.\\
 Furthermore, we have
\begin{equation}\label{mamo}
n\Int_{\xi}^{\xi+\frac{1}{n}}\Int_0^T
\Big|\Frac{\partial^4 w}{\partial x^4}(x,t)\Big|^2\,dx\,dt
 \leq
C\Big(\|f_n\|^2_{L^2(0,T;{\mathcal D}(\partial_x^4))} +
\|w^0\|^2_{{\mathcal D}(\partial_x^4)} + \|w^1\|^2_{V}\Big).
\end{equation}
\end{theorem}
{\bf{Proof of Theorem \ref{h}.}} The proof of (\ref{gg}) and  (\ref{mamo}) can be done by using obvious adaptations of the
proof of (\ref{mi5}) and (\ref{mama}), so it is omitted.\df
 From (\ref{ZZ}) and (\ref{R})
it follows that $\phi(\xi,t)$ is bounded in $F$, after extraction of
a subsequence, $\phi(\xi,t)$ converges in
$L^2(0,T)$ weakly.\\
On the other hand, from (\ref{obser1}), (\ref{aaa}) and
(\ref{aaaa}) we have
 $$\|\phi^0\|_{L^2(0,1)} +
\|\phi^1\|_{V'}=o(n).
$$
Using (\ref{gg}) and (\ref{ggg}) we can easily prove that the mapping
$$
(\phi^0,\phi^1)\in  L^2(0,1)\times V'\rightarrow \Frac{\partial
\phi}{\partial x}(\xi,t) \in H^{-1}(0,T),
$$
is linear and  continuous.\\
Furthermore, we have
$$
\Big\|\Frac{\partial \phi}{\partial x}(\xi,t)\Big\|_{H^{-1}(0,T)} =
o(n).
$$
Now, we prove that $v\in L^{2}(0,T)$ that is
$$
\forall u \in {\mathcal D}(0,T),\,\,\,\,\, |(v,u)|_{{\mathcal
D'}\times {\mathcal D}}\leq c\|u\|_{L^2(0,T)}.
$$
We define the following functions
$$
\Phi(x,t)=\Int_0^t\phi(x,\tau)\,d\tau-\Big(\frac{\partial^4}{\partial
x^4}\Big)^{-1}\phi^1(x),
$$
and
$$
S_n(x,t)= \Int_0^t
\Phi(x,\tau)\,d\tau-\Big(\frac{\partial^4}{\partial
x^4}\Big)^{-1}\phi^0(x).
$$
The functions $\Phi$ and  $S_n$ are solutions of
(\ref{p4o})-(\ref{p6o}) with initial data in $V\times L^2(0,1)$ and
${{\mathcal D}(\partial_x^4)} \times V$
$$
\|\Phi^0\|_{V}+\|\Phi^1\|_{L^2(0,1)}  = o(n),
$$
and
$$
\|S_n^0\|_{{{\mathcal D}(\partial_x^4)} }+\|S_n^1\|_{V} = o(n).
$$
For $u\in {\mathcal D}(0,T)$, we have
{\small\begin{equation*}
 \begin{split}
n\Int_0^T\Int_\xi^{\xi + \frac{1}{n}}\phi(x,t)u(t)\,dx\,dt&=
n\Int_0^T\Int_\xi^{\xi + \frac{1}{n}}S_n(x,t)\frac{\partial^2 u}{
\partial t^2}(t)\,dx\\&= \Int_0^T\Big(S_n(\xi,t)+\frac{1}{2n}
\frac{\partial S_n}{\partial
x}(\xi,t)\Big)\frac{\partial^2u}{\partial t^2}(t)\,dt\\&
+n\Int_0^T\Int_\xi^{\xi+\frac{1}{n}}\frac{\partial^2 u} {\partial
t^2}(t)\Int_\xi^x\Int_\xi^y \frac{\partial^2 S_n}{\partial
z^2}(z,t)\,dz\,dy\,dx\,dt.
\end{split}
\end{equation*}}
Then
\begin{equation}\label{bb}
\Big(S_n(\xi,t)+\frac{1}{2n} \frac{\partial S_n}{\partial
x}(\xi,t),\frac{\partial^2 u}{\partial t^2}(t)\Big)_{{\mathcal
D}',{\mathcal D}}=n\Int_0^T\Int_\xi^{\xi+\frac{1}{n}}
\phi(x,t)u(t)\,dx\,dt-R_n,
\end{equation}
where
$$
R_n=n\Int_0^T\Int_\xi^{\xi+\frac{1}{n}}\frac{\partial^2
u}{\partial t^2}(t)\Int_\xi^x\Int_\xi^y \frac{\partial^2
S_n}{\partial z^2}(z,t)\,dz\,dy\,dx\,dt,
$$
then, we prove that
$$
\Lim_{n\rightarrow \infty}R_n =0.
$$
Using H\"{o}lder's inequality, we have
{\small\begin{equation*}
 \begin{split}
 |R_n|&\leq n\Big\|\frac{\partial^2 u}{\partial t^2}\Big\|_{L^2(0,T)}
 \Big[\Int_0^T\frac{1}{n}\Int_\xi^{\xi+\frac{1}{n}}(x-\xi)
\Int_\xi^x(y-\xi)\Int_\xi^y\Big|\frac{\partial^2 S_n}{\partial
z^2}(z,t)\Big|^2\,dz\,dy\,dx\,dt\Big]^{1/2}\\&
\leq\Frac{1}{\sqrt{8}\sqrt{n}}\Big\|\frac{\partial^2 u}{\partial
t^2}\Big\|_{L^2(0,T)}\|(S_n^0, S_n^1)\|_{{\mathcal
D}(\partial_x^4)\times V}.
\end{split}
\end{equation*}}
Thus
$$
\Lim_{n\rightarrow \infty}R_n =0.
$$
Integrating by part, we get
$$
\Big(S_n(\xi,t)+\frac{1}{2n} \frac{\partial S_n}{\partial
x}(\xi,t),\frac{\partial^2 u}{\partial t^2}(t)\Big)_{{\mathcal
D}',{\mathcal D}} = \Big(\phi(\xi,t)+\frac{1}{2n}\frac{\partial
\phi}{\partial x}(\xi,t),u(t)\Big)_{{\mathcal D}',{\mathcal D}}.
$$
Then
$$
\Big|\Big(\phi(\xi,t)+\frac{1}{2n}\frac{\partial \phi}{\partial
x}(\xi,t),u(t)\Big)_{{\mathcal D}',{\mathcal D}}\Big|\leq
c\|u\|_{L^2(0,T)}+|R_n|.
$$
Passing to the limit as $n$ tends to infinity, we obtain
$$
\Big|(v,u)\Big|\leq c\|u\|_{L^2(0,T)},
$$
which proves that $v$ belongs to $L^2(0,T)$. The proof of lemma
\ref{mm1} is now complete.
\end{proof}
{\bf{Proof of lemma \ref{50}}.}
Passing to the limit in (\ref{AQZ}), we have\\

$
\forall (u^0,u^1,g)\in L^2(0,1)\times V'\times L^2(0,T,L^2(0,1)),
$
\begin{equation*}
\Int_0^T\Int_0^1 g(x,t)\psi(x,t)\,dx\,dt=\Int_0^T v(t)u(\xi,t)\,dt
-\Int_0^1y^0(x)u^1(x)\,dx+\Int_0^1 y^1(x)u^0(x)\,dx,
\end{equation*}
where $u$ is the solution of (\ref{p1})-(\ref{p3}) and $\psi$ is the solution of (\ref{p44oo})-(\ref{p77oo}).\\
 Since the linear form $K$ defined in (\ref{sv}) is meaningful  on $L^2(0,1)\times V'\times L^2(0,T,L^2(0,1))$, it is sufficient to prove that $(K_n)_n$  converges to $K$ on a dense subspace of $L^2(0,1)\times V'\times L^2(0,T,L^2(0,1))$ and, for example,   we consider
 $(u^0,u^1,g_n)\in   L^2(0,1)\times V'\times L^2(0,T,L^2(0,1))$
 $$
 \begin{array}{ll}
 L_n: L^2(0,T, {\mathcal D}(\partial_x^4))\rightarrow \R\\
 \hskip3.50cm u \mapsto n\Int_0^T \Int_\xi^{\xi+
 \frac{1}{n}} \phi(x,t)u(x,t)\,dx\,dt,
 \end{array}
 $$
 are defined and bounded on $L^2(0,T,{\mathcal D}(\partial_x^4))$.\\
 They converge for the weak topology $L^2(0,T,({\mathcal D}(\partial_x^4))')$ to an element $L$ of
 $L^2(0,T,({\mathcal D}(\partial_x^4))')$. In order to determine $L$, we write for $u\in {\mathcal D}(0,T;C^\infty(0,1))$:
 {\small\begin{equation*}
\begin{split}
L_n(u)&= n\Int_0^T\Int_\xi^{\xi+
 \frac{1}{n}}\phi(x,t)u(\xi,t)\,dx\,dt
 \\&+n\Int_0^T \Int_\xi^{\xi+
 \frac{1}{n}}\phi(x,t)\Big(\Int_\xi^x
 \frac{\partial u}{\partial y}(y,t)\,dy\Big)\,dx\,dt.
\end{split}
\end{equation*}
We have already seen that
 $$
 \lim_{n\rightarrow \infty}n\Int_{\xi}^{\xi +
 \frac{1}{n}}\Int_0^T \phi(x,t)u(\xi,t)\,dx\,dt = \Int_0^T  v(t)u(\xi,t)\,dt.
 $$
On the other hand, it is easy to prove, using H\"{o}lder inequality, that for every $u_n\in {\mathcal D}(0,T;C^\infty(0,1))$, we have
$$
 \lim_{n\rightarrow \infty}n\Int_\xi^{\xi+
 \frac{1}{n}}\Int_0^T \phi(x,t)\Big(\Int_\xi^x
 \frac{\partial u}{\partial y}(y,t)\,dy\Big)\,dx\,dt = 0.
 $$
This completes the proof.\df
\begin{remark}
 %\begin{itemize}
%\item[1) ]
By the same method  we can obtain the pointwise controllability
of the  Kirchhoff beam equation
\begin{equation*}
\Frac{\partial ^2 u}{\partial
t^2}(x,t)-\Frac{\partial^{4}u}{\partial t^2\partial x^2}(x,t)+
\Frac{\partial^{4}u}{\partial x^4}(x,t)= v(t)\delta_{\xi},\,\,\,
0<x<1,\,\,\,\ t>0,
\end{equation*}
\begin{equation*}
u(0,t)= \Frac{\partial u}{\partial x}(1,t)= \Frac{\partial^2
u}{\partial x^2}(0,t)= \Frac{\partial^3 u}{\partial x^3}(1,t)=0,
\end{equation*}
\begin{equation*}
u(x,0)=u^{0}(x),\,\,\,\,\ \Frac{\partial u}{\partial
t}(x,0)=u^{1}(x),\,\,\ 0<x<1,\end{equation*}
as a limit of internal exact controllability of
\begin{equation*}
\Frac{\partial ^2 u}{\partial t^2}(x,t)-\Frac{\partial^{4}u}{\partial t^2\partial x^2}(x,t)+
\Frac{\partial^{4}u}{\partial x^4}(x,t)=
g_n(x,t),\,\,\, 0<x<1,\,\,\,\ t>0,
\end{equation*}
\begin{equation*}
u(0,t)= \Frac{\partial u}{\partial x}(1,t)= \Frac{\partial^2
u}{\partial x^2}(0,t)= \Frac{\partial^3 u}{\partial x^3}(1,t)=0,
\end{equation*}
\begin{equation*}
u(x,0)=u^{0}(x),\,\,\,\,\ \Frac{\partial u}{\partial
t}(x,0)=u^{1}(x),\,\,\ 0<x<1,
\end{equation*}

%\item[2) ] The observability inequality of the internal beam equation can be obtained in {\bf\cite{AmTu1}} by the following inequality
%\begin{equation*}\label{1m}
%\Int_0^T \Big|\frac{\partial \phi}{\partial t}(x,t)\Big|^2\,dx \geq C_{n,\pi} (\|u^0\|^2_V+\|u^1\|^2_{L^2(0,1)}),
%\end{equation*}
%and that of the pointwise beam equation can be obtained in {\bf\cite{AmTu2}} by the following inequality
%\begin{equation*}\label{2m}
%\Int_0^T \Big|\frac{\partial \phi}{\partial t}(\xi,t)\Big|^2\,dx \geq C_{\xi} (\|u^0\|^2_V+\|u^1\|^2_{L^2(0,1)}).
%\end{equation*}
%Then, according to {\bf\cite{AmTu2}} we can prove that the pointwise stabilization of beam equation  can be obtained as a limit of internal stabilization of  the same equation.
%\end{itemize}
\end{remark}
%{\bf{Acknowledgments}}\\
 %We thank  very much   Professor Kaïs Ammari
%university of Monastir (Tunisia)  for suggesting us to study this
%difficult problem and for his interesting remarks and comments.

%\end{center}

%\begin{center}
%{\small {\section{Local existence}}} \qquad  \\
%\end{center}

\end{document}